\documentclass[12pt,reqno,a4paper]{amsart}
\usepackage{a4wide,,ifthen,color,amsaddr,mathtools}
\usepackage{tikz} 
\usepackage{todonotes}
\mathtoolsset{showonlyrefs}
\usepackage[all]{xypic}
\usepackage[mathcal]{euscript}

\newcommand{\marginnote}[1]{\ifthenelse{\isodd{\thepage}}{\normalmarginpar}
{\reversemarginpar}\marginpar{\fbox{\parbox{15mm}{\sloppy\footnotesize #1}}}}

\newtheorem{thm}{Theorem}[section]
\newtheorem{corol}[thm]{Corollary}
\newtheorem{lemma}[thm]{Lemma}
\newtheorem{prop}[thm]{Proposition}
\newtheorem{defin}[thm]{Definition}

\theoremstyle{remark}
\newtheorem{rem}[thm]{Remark}
\newtheorem{ex}[thm]{Example}
\newenvironment{remark}{\begin{rem}\rm}{\qee\end{rem}}

\newcommand{\Hom}{\operatorname{Hom}}

\newcommand{\cO}{{\mathcal O}}

\newcommand{\cF}{{\mathcal F}}

\newcommand{\Pt}{{\mathbb P^3}}

\newcommand{\R}{{\mathbb R}}
\newcommand{\C}{{\mathbb C}}
\newcommand{\Z}{{\mathbb Z}}

\newcommand{\Q}{{\mathbb Q}}

\newcommand{\qee}{\mbox{\hspace{0.2mm}}\hfill$\triangle$}

\newcommand{\var}{{{\mathbb P}_\Sigma}}

\begin{document} 
\title[Noether-Lefschetz locus of surfaces in toric 3-folds]{The Noether-Lefschetz locus\\[8pt] of surfaces in toric threefolds}
\author{\vskip-12pt\small Ugo Bruzzo$^{\S\ddag}$ and Antonella Grassi$^\P$}
\address{\rm $^\S$ Area di Matematica, Scuola Internazionale Superiore di Studi   \\ Avanzati (SISSA),
Via Bonomea 265, 34136 Trieste, Italia}
\address{\vskip-12pt\rm $^\ddag$ Istituto Nazionale di Fisica Nucleare, Sezione di Trieste}
\address{\vskip-12pt\rm $^\P$ Department of Mathematics, University of Pennsylvania,\\
David Rittenhouse Laboratory, 209 S 33rd Street,\\ Philadelphia, PA 19104, USA}

\email{bruzzo@sissa.it, grassi@sas.upenn.edu}
\thanks{Support for this work was provided by the NSF Research Training Group Grant
DMS-0636606, the Brazilian CNPq grant 310002/2015-0, {\sc prin} ``Geometria delle variet\`a  algebriche,''  and {\sc gnsaga-in}d{\sc am}. U.B. is a member of the {\sc vbac} group.}
\date{\today}
\begin{abstract}   The Noether-Lefschetz theorem asserts that  any curve in a very general surface $X$  in $\mathbb P^3$  of degree $d \geq 4$  is  a restriction of a surface in the ambient space, that is, the Picard number  of $X$ is   $1$. We proved previously that under some conditions, which
replace the condition  $d \geq 4$, {a very general} surface   in a simplicial toric threefold $\var$ (with orbifold singularities) has the same Picard number as $\var$.   Here we define the Noether-Lefschetz  loci  of  quasi-smooth surfaces in  $\var$  in a linear system of a Cartier ample divisor with respect to a (-1)-regular, respectively 0-regular, ample Cartier divisor, and   give bounds on their codimensions.   We also study the components of the Noether-Lefschetz loci   which contain  a line, defined as a rational curve which is minimal in a suitable sense.
\end{abstract}

\maketitle

\section{Introduction}
 The Noether-Lefschetz theorem states that  any curve in a very general surface $X$  in $\mathbb P^3$  of degree $d \geq 4$  is  a restriction of a surface in the ambient space, namely the Picard number  of $X$ is   $1$  (a point is very general if it lies outside a countable union of closed subschemes of positive codimension). 
The Noether-Lefschetz locus is the locus where the Picard number is greater than $1$.   Green  \cite{Green2,Green1} and Voisin \cite{Voisin88Precision}  proved that if ${\mathcal S}_d $ is the locus of the degree $d$ surfaces in $\mathbb P^3$ whose   Picard number  is not $1$,   every component of ${\mathcal S}_d $ has codimension $\geq d-3$, with equality for the components of surfaces containing a line. This result has been generalized by Otwinowska  \cite{OtwinowskaJAG} to hypersurfaces in $\mathbb P^n$; also Lopez-Maclean  obtained a bound on the codimension of  the Noether-Lefschetz locus of surfaces associated with a very ample line bundle in  smooth threefolds, under some conditions \cite{LopMac}.

 We continue the analysis of \cite{BG1}, and consider a  projective toric threefold  $\var$ with orbifold singularities; $\var$ is associated with a $3$-dimensional complete simplicial fan $\Sigma$ and is $\mathbb Q$-factorial  (see Section \ref{divisors}).   Let  $\beta$ be a nef (numerically effective) class in  the class group $Cl(\var)$ of Weil divisors modulo rational equivalence and consider a  surface $X$ in $\var$ whose class (degree) in $Cl(\var)$ is $\beta$. If $X$ is general, it is quasi-smooth, that is, its only singularities are those inherited from $\var$ (Section \ref{hypersurfaces}).   
 
 Let $\mathcal M_\beta$ be the moduli space of surfaces in $\var$ of degree $\beta$ modulo automorphisms of $\var$.   In \cite{BG1}, see Theorem \ref{BG1} in Section \ref{codimsec},
  we proved that for $\beta$ ample and $-\beta_0$ the canonical class of $\var$,  if the multiplication morphism
 \begin{equation} \label{eqabove}
  R(f)_\beta\otimes R(f)_{\beta-\beta_0} \to R(f)_{2\beta-\beta_0}
   \end{equation}
is surjective,     very general points of $\mathcal M_\beta$ correspond to  surfaces whose Picard number equals the Picard number of $\var$; here $R(f)$ is the Jacobian ring, see Section \ref{hypersurfaces}.
If $\var=\Pt$ and $\beta=d \geq 4$,   or equivalently $\beta- \beta_0 $ is nef,
 the morphism in \eqref{eqabove} is always surjective.
Also, the morphism is surjective whenever  $\beta-\beta_0$ is trivial, that is, $X$ is a K3 surface in a Fano threefold  $\var$.
 If $\var$ is an Oda variety, that is,  the sum of two polytopes associated with a nef and an ample divisor is equal to  their Minkowski sum, the multiplication map in \eqref{eqabove} is always surjective, see  Section \ref{Oda}. 

If we  write $\beta-\beta_0=n \eta$ for  an ample Cartier  primitive class $\eta$,  the condition  $n \geq 0$  generalizes the classical condition $d \geq 4$. We define the Noether-Lefschetz locus with respect to $\beta$ to be the closed subscheme
$U_\eta(n)$ of $\mathcal M_\beta$ corresponding to quasi smooth surfaces whose Picard number is strictly larger than that of $\var$ (Definition \ref{defvarie}).  
In particular we have an  upper bound on the codimension of any  irreducible component on the Noether-Lefschetz locus:

\smallskip
 
\noindent{\bf Proposition \ref{above}.} {\em 
\begin{equation}  \operatorname{codim} U_{\eta}(n) \leq h^{2,0}(S)=h^0(\var, \cO_\var(n \eta)), \end{equation}
where $S$ is a quasi-smooth surface in the linear system $\beta$.
}

\smallskip

The classical proof  of the codimension of the Noether-Lefschetz locus for $\var=\Pt$
relies implicitly  on the fact that $ \eta= \cO_\Pt(1)$ is $(-1)$-regular and any line bundle of degree $ d \geq 4$ is 0-regular. However, in Theorem \ref{A3} we show that $\mathbb P^n$ is the only    toric $n$-fold with an ample $(-1)$ regular line bundle. So we consider toric varieties with a 0-regular ample line bundle. Our proof generalizes the arguments of Green in \cite{Green2, Green1} and  relies on vanishing theorems and dualities for toric varieties, as well  as   the  Castelnuovo-Mumford regularity of certain bundles.

 In Section \ref{codimsec} we bound the codimension of the Noether-Leschetz locus. In Theorem \ref{mainR}  we prove
 
\smallskip\noindent
{\bf Theorem \ref{mainR}}. {\em 
Let $\var$ be a   simplicial toric threefold, $\eta$  an ample {primitive} Cartier class, $\beta_0=-K_\var$, $\beta\in \operatorname{Pic}(\var)$   an ample Cartier class that satisfies $\beta - \beta_0 = n\eta$ for some $n \geq 0$. Assume  that $\beta$ is 0-regular with respect to $\eta$.
If $\eta$ is (-1)-regular,  
then
\begin{equation}\label{inequalityA3}\operatorname{codim} U_{\eta}(n) \geq n+1.\end{equation}
 If $\eta$ is 0-regular, then
 \begin{equation}\label{inequalityA2}\operatorname{codim} U_{\eta}(n)\geq n.\end{equation}
}

\smallskip\noindent
{\bf Corollary \ref{Fano}.} 
{\em
 Let $\var$ be a  simplicial Fano toric threefold, $\eta$  a  {primitive} nef divisor, $\beta_0=-K_\var$, $\beta\in \operatorname{Pic}(\var)$   an ample Cartier class that satisfies $\beta - \beta_0 = n\eta$ for some $n \geq 3$.  \\
If $\eta$ is $(-1)$-regular 
then
\begin{equation}\label{inequalityA3}\operatorname{codim} U_{\eta}(n) \geq n+1.\end{equation}
 If $\eta$ is 0-regular, then
 \begin{equation}\label{inequalityA2}\operatorname{codim} U_{\eta}(n)\geq n.\end{equation}
}

\smallskip
We prove a similar result  for Oda varieties in Theorem \ref{mainO}.

We consider various examples in Section \ref{examples}; in
Section \ref{smallCod} we consider  the components of the loci $U_\eta(n)$ which contain  a line, defined as a rational curve that is ``minimal'' in a suitable sense (i.e., its intersection with the ample class $\eta$ is 1). We show that the codimension of these components  is $n+1$, as in the classical case. To do that, we prove a Severi-type vanishing theorem, and, in a singular case, we study the Hilbert schemes of lines. 

\bigskip
{\bf Acknowledgements.} We thank R.\ Lazarsfeld,  F.\ Perroni  and C.\ Voisin for useful discussions, and the  referee for  many helpful suggestions. This paper was partly written while the first author was visiting UFSC in Florian\'opolis, Brazil,  supported by the  CNPq grant 310002/2015-0 and the University of Pennsylvania. He gratefully acknowledges CNPq and the Department of Mathematics at UFSC, especially its Geometry group, for the hospitality and the Department of Mathematics at the University of Pennsylvania.

\bigskip\section{Preliminaries}\label{generalities}

\subsection{Divisors  in toric varieties and the Picard number}\label{divisors}
  Let $M$ be a free abelian group of rank $r$, let $N=\Hom(M,\Z)$, and  $N_\R=N\otimes_\Z\R$.
  Let $\Sigma$ be a  rational simplicial complete  $r$-dimensional fan in $N_\R$. It  defines a complete toric variety $\var $ of dimension $r$  having only  Abelian quotient singularities. 
$\var$  is said to be {\it simplicial}, and   is  an orbifold.  

Let  $Cl(\var)$ be the group of Weil divisors in $\var$ modulo rational equivalence,
and  $\operatorname{Pic}(\var)$  the group of line bundles on $\var$ modulo isomorphism; both are
finitely generated Abelian groups, and $\operatorname{Pic}(\var)$ is   free. Since $\Sigma$ is simplicial,
 $\var$ is $\Q$-factorial,   i.e., the natural inclusion
$\operatorname{Pic}(\var) \hookrightarrow Cl(\var)$ becomes an isomorphism after
tensoring by $\Q$. The rank of the two groups  is also the {\it Picard number}, the
rank of the N\'eron-Severi group of $\var$.
Recall that the  N\'{e}ron-Severi group of a variety $Y$
 is  the image of the Picard group in the second cohomology group with integer coefficients.
One can define its rank    as 
$\dim_{\Q} NS(Y) \otimes  _{\Z} \Q = \dim_{\Q} (H^{2}(Y,\Q) \cap  H^{1,1}(Y,\C))$.

Let $S$ be the   {\it Cox ring}, that is, the algebra over $\C$ generated by the homogeneous coordinates associated with the rays of $\Sigma$; so $S = \C[x_\rho ]$, where $\rho$ runs over the rays of the fan \cite{Cox95}.   This is a generalization of the coordinate ring of $\var={ß\mathbb P}^n$,  in which case the Cox ring is    $\C[x_0, \cdots, x_n]$.  Recall that  each ray $\rho$  determines a Weil divisor $D_\rho$; a monomial $ \prod x_\rho^{a_\rho}$ determines a divisor $D= \sum_\rho a_{\rho}D_\rho$ and $\deg(D)= \deg (\prod x_\rho^{a_\rho}) \in Cl(\var)$ induces a grading  of  $S =\oplus_{\gamma\in Cl(\var)}S_\gamma$.

 \subsection{Surfaces in $\var$}\label{hypersurfaces} Let  $D$ be a  nef divisor of class $\beta$  in $Cl(\var)$; then  $\cO_\var(D)$ is generated by its global sections \cite[Th.~1.6]{Mavlyutov-semi}.
 Let $X$ be a  surface  in $\var$ whose class  in $Cl(\var)$ is $\beta$;
if $X$ is general, it is {\it quasi-smooth}, that is, its only singularities are those inherited from $\var$ \cite[Lemma 6.6, 6.7]{MulletToric}. In particular, $X$ is an orbifold if $\var$ is simplicial.

Let $f$ be a section of the line bundle
$\cO_\var(D)$  such that $X$ is the locus $f=0$. The {\it 
 ideal } $J(f)$ is the ideal in $S$ generated by the derivatives of $f$, and the {\it Jacobian ring} $R(f)$ is defined as $R(f)=S/J(f)$. It is naturally graded by the class group $Cl(\var)$.

 \subsection{The dualizing sheaf and vanishing theorems} \label{Zariski} Let $\mathbb P_{\Sigma,0}$ be the smooth locus of $\var$, and let $j\colon \mathbb P_{\Sigma,0} \to \var$ be the corresponding open immersion. The sheaf $\widehat\Omega_\var^p$ of Zariski $p$-forms on  $\var$ is the sheaf
$j_\ast\Omega^p_{\var,0}$, where $ \Omega^p_{\var,0} $ is the sheaf of $p$-differentials on $\mathbb P_{\Sigma,0}$. The sheaf $\omega_\var = \widehat\Omega_\var^r$ is the dualizing sheaf of $\var$. It is a reflexive coherent sheaf of rank 1, and determines a class in $Cl(\var)$ that we shall denote $-\beta_0$. If $\beta$ is the class of a Weil divisor, we shall denote by  $\cO_\var(\beta-\beta_0)$ the sheaf $\cO_\var(\beta)\otimes\omega_\var$.

If $\beta$ is an ample class in $\operatorname{Pic}(\var)$,  the Bott-Danilov-Steenbrink vanishing theorem \cite[Ch.~3]{Oda88} says:
\begin{equation}\label{bott}
H^q(\var,\cO_\var(\beta)\otimes\widehat\Omega_\var^p) = 0
\end{equation}
for   $1\le q \le r$ and $0\le p \le r$.

Toric Serre duality (\cite{CoxLSh}, Theorem 9.2.10) then implies

\begin{equation}\label{bottdual}
H^q(\var,\cO_\var(-\beta)\otimes\widehat \Omega_\var^{p}) = 0
\end{equation}
for   $0\le q\le r-1$ and $0\le p \le r$.

If $\beta$ is Cartier and nef, Mavlyutov's vanishing theorem \cite{Mavly-cohom} 
holds since $\var$ is complete and simplicial,  and one has
\begin{equation}\label{MavlyutovNef}
H^q(\var,\cO_\var(\beta)\otimes\widehat\Omega_\var^p) = 0
\end{equation}
for  $q> p$ or  $p > q+ \dim P_{\beta}$, where
$P_{\beta}$ is the polytope associated with the line bundle
$\cO_\var(\beta)$ (see also  \cite{CoxLSh}, Theorem 9.3.3). For $p=0$ this reduces
to Demazure's  vanishing theorem \cite[Thm.~9.2.3]{CoxLSh}, that is,
$H^q(\var,\cO_\var(\beta) ) = 0 $
for $q > 0$, under the only hypothesis that $\beta$ is nef.

 
\bigskip
\section{Regularity and Oda varieties}\label{NefCMReg}

The    proofs, based on the infinitesimal variation of Hodge structure, of the  Noether-Lefschetz theorem,  and of the  lower bound for the codimension of the Noether-Lefschetz locus in the moduli space of surfaces in $\Pt$,   use three fundamental ingredients:
the (-1)-regularity of the hyperplane bundle $\cO_{\Pt}(1)$, the fact that for any $n,m\in\mathbb N$, the multiplication morphism $S_n\otimes S_m \to S_{n+m}$
is surjective, and  various  cohomology vanishings.
 While these properties do not hold in general for a toric variety, we find that suitable generalizations of the above hypotheses are satisfied by a large class of toric varieties.
 These generalizations are the key for  the proof of Theorem \ref{main}. 
 
\subsection {Castelnuovo-Mumford regularity for toric varieties}\label{CMreg}
Let $X$ be a projective variety and $L$ an ample and globally generated  line bundle on it.
\begin{defin} A coherent $\cO_X$-module $\cF$ is $m$-regular with respect to $L$ if
$$H^q(X,\cF\otimes L^{m-q})=0$$ for all $q>0$.
\end{defin}

If $L$ is an ample and globally generated  line bundle which is $m$-regular with respect to itself, we shall just say that it is $m$-regular.
A line bundle on a complete toric variety is nef if and only if it is globally generated \cite[Th.~1.6]{Mavlyutov-semi}, and we have the following:

\begin{thm}[\cite{Lazarsfeld}, Thm~1.8.5] \label{laza}  Let $\var$ be a projective  toric variety. 
If a locally free $\cO_\var$-module $\cF$ is $m$-regular with respect to an ample line bundle $L$, then for all $k \geq 0$,
\begin{enumerate} \item
 $\cF \otimes L^{m+k}$ is generated by  global sections; \item the map
\begin{equation}\label{surj} H^0(\var,\cF \otimes L^m) \otimes H^0(\var,L^k) \to H^0(\var,\cF\otimes L^{k+m})
\end{equation}
is surjective;
\item $\cF$ is $(m+k)$-regular.
\end{enumerate} \end{thm}


\begin{corol}
If $L$ is 0-regular, for every $m$-regular locally free $\cO_\var$-module $\cF$, there is a projective resolution
\begin{equation}\label{resolution}
\dots\to  \bigoplus (L^{-m-1}) \to  \bigoplus (L^{-m}) \to \cF \to 0 \,.\end{equation}
\end{corol}

\begin{proof} 
The proof of \cite[Prop.~1.8.8]{Lazarsfeld}, together with Theorem \ref{laza},  can be applied.
\end{proof}

\begin{rem} Note that $L^{-m-j}$ is $(m+j+1)$-regular if $L$ is 0-regular, and $(m+j)$-regular if $L$ is 
$(-1)$-regular. \end{rem}

 If $\var = \mathbb P^n$ and $\cF$ is locally free and $1$-regular with respect to $\cO_{\mathbb P^n}(1)$, then $\cF^{\otimes p}$ is $p$-regular
a consequence of the fact that $\cO_{\mathbb P^n}(1)$ is (-1)-regular \cite[Prop.~1.8.9]{Lazarsfeld}. The same holds on $(\var,  \eta)$ with $\eta$ ample and (-1)-regular:

\begin{prop}\label{tensor}  Let  $\var$ be a    toric variety, and $L$ a $(q-1)$-regular  ample  line bundle on it,
 $q\ge 0$.
 Il $\cF$ is locally free and $1$-regular with respect to $L$, then $\cF^{\otimes p}$ is $(p+q)$-regular.
 \label{regular}
\end{prop}

\begin{proof} It is enough to show the results when $L$ is (-1) regular.
Since $\cF$ is locally free, by tensoring $p$ copies of the resolution \eqref{resolution} for $m=1$ we obtain a resolution
$$ \dots \bigoplus(L^{-p-1}) \to  \bigoplus (L^{-p}) \to \cF^{\otimes p}  \to 0 \,.$$
By splitting this into short exact sequences and taking cohomology one gets the result.
\end{proof}

\begin{prop}\label{conditions}  Let $\var$ be an $r$-dimensional   toric   variety,
and $\eta$ an ample Cartier divisor.
\begin{enumerate}\setlength\itemsep{3pt}
\item  $\eta$  is $0$-regular if and only if  $H^0(\var,\cO_\var((r-1)\eta-\beta_0))=0$.
\item $\eta$  is $(-1)$-regular if and only if  $H^0(\var,\cO_\var(r\eta-\beta_0))=0$.
\end{enumerate}
\end{prop}
\begin{proof} Follows from Serre duality and the Bott-Danilov-Steenbrink vanishing.
\end{proof}

 \begin{thm} \label{onlyGore} The only Gorenstein toric threefolds with Picard number 1 that carry a 0-regular ample line bundle are $ \mathbb P[1,1,2,2]$ and $\mathbb P^3$.
\end{thm}

\begin{proof}  Write $\beta_0=\ell \eta$, where $\eta$ is an ample Cartier divisor; then $2\eta- \beta_0=(2-\ell) \eta$, and  $\eta$ is 0-regular if and only if $ \ell \geq 3$ (Proposition \ref{conditions}).     Then $\var$  is  Fano and $\ell$ is the index of  $\var$ as in \cite{OgataZhao}. Ogata and Zhao prove that $ \ell =3$ if  $ \var \neq \mathbb P^3$ \cite[Theorem~1]{OgataZhao}.  The statement then follows from Theorem 3 in \cite{OgataZhao}.
\end{proof}

\begin{thm}\label{A3} The only $r$-dimensional    toric variety ($r\ge 2$)  with an ample line bundle which is (-1)-regular is the projective space $\mathbb P^r$.
\end{thm}
\begin{proof} Let $\eta$ be an ample Cartier divisor  in the variety under consideration. By  Proposition   \ref{conditions}(ii) we have that the line
bundle $\mathcal O_\var(r\eta-\beta_0)$ has no sections. The claim then follows from 
Theorem 1 in \cite{OgataZhao} (or, for the same statement in purely combinatorial terms,   \cite[Prop.~1.4]{BatyrevNill}).
\end{proof}

\subsection{Oda varieties}\label{Oda}

\begin{defin}\label{A1}  A toric variety $\var$ is an Oda variety if the multiplication morphism $S_{\alpha_1 }\otimes S_{\alpha_2} \to S_{\alpha_1+\alpha_2}$
is surjective when the classes $\alpha_1$ and $\alpha_2$ in $\operatorname{Pic}(\var)$ are
ample and nef, respectively.
\end{defin}

The  question of the surjectivity of this map was  posed by  Oda in \cite{OdaUnpubl}  under more general  conditions.
Note that the conjecture is still open even for smooth projective toric varieties. 
This assumption  can be stated in terms of the Minkowski sum of polytopes, because the  integral points of a polytope associated with a line bundle correspond to   sections of the line bundle. Definition \ref{A1} says  that  the sum $P_{\alpha_1}+ P_{\alpha_2}$  of the polytopes associated with the line bundles ${\mathcal O}_{\var}(\alpha_1)$ and ${\mathcal O}_{\var}(\alpha_2$) is equal to  their Minkowski sum, that is $P_{\alpha_1+\alpha_2}$, the polytope associated with the line bundle ${\mathcal O}_{\var}(\alpha_1+\alpha_2)$. 
The Oda varieties are characterized by Property 2.2 by Ikeda in  \cite{Ikeda09}, as $\alpha_1$ is nef if and only if it is globally generated.  
Ikeda  also shows the following:
\begin{thm} \label{Ikeda}\cite[Corollary ~ 4.2]{Ikeda09}
\begin{enumerate} 
\item A smooth toric variety  with Picard number $2$ is an Oda variety.
\item  The total space of a toric projective bundle on an Oda variety is also an Oda variety. 
\end{enumerate}
\end{thm}
 We have: 
\begin{prop}\label{corA1}
Let $\var$ be an $n$-dimensional projective toric variety. If $\operatorname{Pic}(\var) = \mathbb Z$  and its ample generator $\eta$
is 0-regular, then $\var$ is an Oda variety.
\end{prop} 
\begin{proof}
If $\alpha_2$ is a nef class in $\operatorname{Pic}(\var)$, then $\alpha_2= \ell \eta$
with $\ell \ge 0$. If $\ell = 0$  the condition in Definition  \ref{A1}  is empty, so we can assume $\ell \ge 1$.
If we set $\mathcal F = \mathcal O_\var(\alpha_2)$ and $L= \mathcal O_\var(\eta)$,
equation  \eqref{surj} in Theorem \ref{laza} becomes the wanted surjectivity condition, provided that $\alpha_2$ is 0-regular (with respect to $\eta$).  We will show that 
 $$ H^i(\var, \mathcal O_\var((\ell - i) \eta ) )= 0$$
 for $\ell \ge 1$ and $i>0$. For $0 < i < n $  all vanishings are a consequence of the Bott-Danilov-Steenbrink vanishing \eqref{bott}, possibly using Serre duality \eqref{bottdual}. For $i=n$ and 
 $\ell \ge n$ one uses again the Bott-Danilov-Steenbrink vanishing. For $i=n$ and $\ell =1 $
 one uses Serre duality and Proposition \ref{conditions}; the cases $i=n$ and $1 < \ell < n $ follow {\em a fortiori.} 
\end{proof}

\subsection{Other vanishings}\label{othervanishings}
 Let $\var= \Pt$. Then   any ample divisor $\beta$ of degree at least  $4$ is 0-regular; also, $\beta - 2 H$  is nef.  The vanishings
$h^1(\Pt, \cO_\Pt (\beta - \eta ))=h^2(\Pt, \cO_\Pt (\beta - \eta ))=h^2(\Pt, \cO_\Pt (\beta - 2\eta ))=0$
 follow from either fact. These vanishings are used in the proof of the estimate of the codimensions of the Noether-Lefschetz locus for $\Pt$. In our more general setting, the proof of Theorem \ref{main}  
 will require analogous vanishing conditions: \begin{equation}\label{extra}h^1(\var, \cO_\var (\beta - \eta ))=h^2(\var, \cO_\var (\beta - \eta ))=h^2(\var, \cO_\var (\beta - 2\eta ))=0. \end{equation} 
So our next step is to study these vanishing conditions.
\begin{prop}\label{extranef}  Let  $\var$ be a  complete simplicial toric variety, $\beta$  and $\eta$  nef divisors. If $\beta - 2 \eta$ is nef, then the vanishings in equation \eqref{extra} are satisfied.
\end{prop}
\begin{proof}  It follows from Demazure's vanishing theorems (the vanishings in \eqref{MavlyutovNef} for $p=0$).
\end{proof}

\begin{prop}\label{extrareg}  Let  $\var$ be a    simplicial toric variety, $\beta$  and $\eta$    Cartier divisors, with $\beta$ nef and $\eta$ ample.  Assume that  $\beta $ is  0-regular with respect to $\eta$. Then:
\begin{enumerate} \item  the vanishings in equation \eqref{extra} are satisfied;
\item  the multiplication morphism $S_\beta \otimes S_{k \eta} \to S_{\beta+ k \eta}$ is surjective
for all $k\ge 0$.
\end{enumerate}
\end{prop}
\begin{proof}  In fact $\beta$ is also $1$-regular by (iii) in Theorem \ref{laza} and the vanishings follow from the definition of regularity.  Part (ii) in Theorem \ref{laza} implies the surjectivity of the multiplication map. \end{proof}

\begin{corol} Let $\var$ be a Fano toric simplicial threefold, $\eta$ an ample Cartier class,
and $\beta=\beta_0+n\,\eta$. If $n\ge 2$, then $\beta-2\eta$ is ample. If $n\ge 3$, then $\beta$ is 0-regular with respect to $\eta$. \label{corFano}
\end{corol}


\bigskip
\section{Noether-Lefschetz loci}\label{codimsec}

In this section we continue the analysis of \cite{BG1} for $X$   a general surface in   a complete simplicial toric variety $\var$ corresponding to an ample Cartier divisor; $X$ is quasi-smooth, see Sections \ref{divisors} and \ref{hypersurfaces}.  
Let $\mathcal M_\beta$ be the moduli space of surfaces in $\var$ of degree $\beta$ modulo automorphisms of $\var$ (for a precise definition see \cite{BaCox94}).  
 We want to establish a lower bound for the codimension of  the closed subscheme
of $\mathcal M_\beta$ corresponding to surfaces whose Picard number is strictly larger than that of $\var$. The codimension will depend on a fixed class  $\eta\in \operatorname{Pic}(\var)$, an ample  {\it primitive} Cartier class (primitive means that it is not a multiple of an ample class).
  If $\var=\Pt$  the condition that $D+K_\var$ is  nef   in the theorems below is the classical condition  $\deg(D)  \geq 4$.  
 \begin{remark} 
  Note that the condition $\beta - \beta_0 = n\eta$ implies that the anti-canonical class $\beta_0$ is Cartier,
namely, $\var$ is necessarily Gorenstein. 
\end{remark} 
In \cite{BG1} we proved:
\begin{thm}\label{BG1}\cite{BG1} Let $\var$ be a 3-dimensional   simplicial toric variety,  $D$ an ample Cartier divisor on $\var$,  $X$ a very general (quasi-smooth) surface in the linear system $\vert D\vert$, $\beta=\deg D$ and $\beta_0 = - \deg K_{\var }$.
If  the morphism 
  \begin{equation}\label{mult} 
  R(f)_\beta\otimes R(f)_{\beta-\beta_0} \to R(f)_{2\beta-\beta_0}
   \end{equation}
is surjective,     $X$ and $\var $ have the same Picard number.
\label{picard}\end{thm} 

Actually when we wrote \cite{BG1} we were unaware of \cite{RavSri09} where a general result is proved using different techniques.

\begin{thm}\label{BGO} Let $\var$ be a 3-dimensional   simplicial toric Oda variety,  $D$ an ample Cartier divisor on $\var$ such that $D+K_\var$ is nef,  $X$ a very general (quasi-smooth) surface in the linear system $\vert D\vert$.  Then  $X$ and $\var $ have the same Picard number.
\end{thm} 
\begin{proof}  Let $\beta=\deg D$, $\beta_0 = - \deg K_{\var }$ and note that  the multiplication morphism \eqref{mult}  is  surjective on an Oda variety. \end{proof}

\begin{thm}\label{BGR}  Let $\var$ be a 3-dimensional   simplicial toric variety,  and $D$ an ample Cartier divisor on $\var$. Let $\beta = \deg D$, and assume that $\beta = \beta_0 + n\,\eta$ 
with $n \ge 0$ for an ample Cartier class $\eta$. If  $\beta$ is 0-regular with respect to $\eta$ and  $X$ is a very general (quasi-smooth) surface in the linear system $\vert D\vert$, then  $X$ and $\var $ have the same Picard number.
\end{thm} 
\begin{proof}  The map  \eqref{mult} is surjective because  
 $S_{\beta }\otimes S_{n\eta} \to S_{\beta+ n \eta}$ is   by Proposition \ref{extrareg}.
\end{proof}

 \begin{defin}[Noether-Lefschetz locus]
  \label{hypo}  Let $\var$ be a  simplicial toric threefold, $\eta$  a  {primitive} ample Cartier class, $\beta_0=-K_\var$  and $\beta\in \operatorname{Pic}(\var)$   an ample Cartier class that satisfies $\beta - \beta_0 = n\eta$ for some $n \geq 0$. Assume that   very general surfaces in the linear system $\beta$  have the same Picard number of $\var$.
 
  $U_{\eta}(n)$  is   the closed subscheme of $\mathcal M_\beta$, the moduli space of surfaces in $\var$ of degree $\beta$ modulo automorphisms,  corresponding to surfaces whose Picard number is strictly larger than that of $\var$.\label{defvarie}
 \end{defin}

\begin{prop}\label{above}  The codimension of any  irreducible component
of  ${U}_\eta(n)$ is bounded from above by $h^{2,0}(S)=h^0(\var, \cO_\var(n \eta))$, where $S$ is a quasi-smooth surface in the linear system $\beta $.
\end{prop}

\begin{proof} The  classic inequality $\operatorname{codim}{U}_\eta (n) \leq h^{2,0}(S)$  for
surfaces in $\mathbb P^3$ \cite[pp.\ 71-72] {CGGH} relies on the fact that a 2-cycle in a smooth projective surface $S$ is algebraic if and only if it is orthogonal to $H^{2,0}(S)$. In our case, a general surface $S$ in the linear system $\beta$ is quasi-smooth and is a projective orbifold. However the argument may be repeated in this case by taking a resolution of singularities $\rho\colon S' \to S$.  The induced morphism $\rho^\ast \colon H^2(S,\mathbb C) \to H^2(S',\mathbb C)$ is injective and the Hodge decomposition of   $H^2(S',\mathbb C) $ induces a Hodge decomposition of $H^2(S,\mathbb C)$ \cite{SteenVanishing}. A class
in $H^{1,1}(S)\cap H^2(S,\mathbb Z)$ regarded in $H^{1,1}(S') \cap H^2(S',\mathbb Z)$
is algebraic and so is algebraic in $S$ as well (the other implication is obvious). 
To see that $h^{2,0}(S)=h^0(\var, \cO_\var(n \eta))$ one uses the fundamental sequence of the 
 divisor $S$, with Serre duality and the 
acyclicity of the sheaf $\cO_\var$. Note that the canonical class of $S$ is Cartier by the adjunction formula, since $\beta_0$ and $\beta$ are Cartier.
\end{proof}

 
\subsection{Some preliminary Lemmas.}\label{preliminary}
 To give an estimate on the codimension of ${U}_\eta(n)$ we need some preliminary Lemmas.
Let $\var$ be a   simplicial toric   threefold.  

In the following we assume that $\eta$ is a  {primitive} ample 0-regular  Cartier class,  and $\beta\in \operatorname{Pic}(\var)$   a  {\it nef} Cartier class that satisfies $\beta - \beta_0 = n\eta$ for some $n \geq 0$, with $\beta_0=-K_\var$.
We also assume  that the multiplication morphism $S_\beta\otimes S_{k\eta} \to S_{\beta+k\eta}$
is surjective for all $k\ge 0$,  and    that  $h^1(\var, \cO_\var (\beta - \eta ))=h^2(\var, \cO_\var (\beta - \eta ))=h^2(\var, \cO_\var (\beta - 2\eta ))=0$ (see Section \ref{othervanishings}).

Recall that $S_\beta=H^0(\var,\cO_\var(\beta))$. Since $\beta$ is nef, it is globally generated, and we have an exact sequence
\begin{equation} \label{M0} 0\to M_0  \to S_\beta\otimes \cO_\var \to  \cO_\var (\beta) \to 0 \end{equation}
where $M_0$ is locally free.  Then:

\begin{lemma} \label{l0}  $H^q(\var,M_0(k \eta)) = 0$ for
all  $q\ge 1$ and $k+q\ge 1$. In particular, $M_0$ is 1-regular.
\end{lemma}

\begin{proof}  
We twist the exact sequence \eqref{M0} by $\cO_\var(k\eta)$ and take a segment of the long exact sequence of cohomology:
\begin{multline} \label{long}
\dots
\to  S_\beta \otimes H^{q-1}(\var,\cO_\var(k\eta))
\stackrel{m}{\relbar\joinrel\longrightarrow}  H^{q-1}(\var,\cO_\var(\beta+k\eta) )\\
\to H^{q}(\var,M_0(k\eta))
\to  S_\beta \otimes H^{q}({\var},\cO_\var(k\eta) )\to \dots
\end{multline}
We consider separately  the cases $q=1,2,3$.

$q=1$. The morphism $m$  is the multiplication morphism $S_\beta\otimes S_{k\eta} \to S_{\beta+k\eta}$ which is surjective by hypothesis. Moreover $H^{1}({\var},\cO_\var(k\eta) )=0$
by the Demazure vanishing theorem  for $k\ge 0$.

$q=2$. If $k \ge 0$,  $H^{1}({\var},\cO_\var(\beta+k\eta) )=H^{2}(\cO_\var (k\eta))=0 $ by the Demazure vanishing theorem, since $\beta+k\eta$  and $\eta$ are nef. 
For  $k=-1$, $H^{1}({\var},\cO_\var(\beta-\eta) )=0$ by assumption, and $H^{2}(\var,\cO_\var(-\eta)=0$ because $\eta$ is 0-regular.

$q=3$.  
For $k \geq 0$, $\beta+k\eta$ is nef so $H^{2}({\var},\cO_\var(\beta+k\eta) ) = H^{3 }({\var},\cO_\var(k\eta) ) =0$  again by Demazure's vanishing. For $k=-2, -1$ we have $H^{2}({\var},\cO_\var(\beta+k\eta) )=0$ by hypothesis and 
$H^3(\var, \cO(k\eta))=0$ because $\eta$ is 0-regular and by Theorem \ref{laza}.
 \end{proof}

\begin{lemma}\label{l1} Assume that  $M_0$ is 1-regular.
If $\eta$ is (-1)-regular, the vector bundle $\Lambda^pM_0$ is $p$-regular with respect to $\eta$ for all $p\ge 1$, that is $H^q(\var, \Lambda^p M_0((p-q) \eta)) = 0$ for  $q>0$.

Similarly, if $\eta$ is 0-regular, then   $\Lambda^pM_0$ is $(p+1)$-regular, that is $H^q(\var, \Lambda^p M_0((p+1-q) \eta)) = 0$ for  $q>0$.
\end{lemma}
\begin{proof} 
$\Lambda^pM_0$ is a direct summand of $M_0^{\otimes p}$, so that the Lemma follows from   Proposition \ref{regular}.
\end{proof}

Let $W$ be a base point free subspace of $S_\beta=H^0(\var,\cO_\var(\beta))$, and let
$$ W = W_c \subset .... \subset W_0 = S_\beta $$
be a flag of subspaces of $S_\beta$ whose quotients  $W_i/W_{i+1}$ are all of dimension $1$ (thus, $c=\dim S_\beta-\dim W$).  All subspaces $W_i$ are base point free (since they
contain $W$), so that there are exact sequences
\begin{equation} \label{M1}
0 \to M_i \to W_i\otimes \cO_\var \to  \cO_\var(\beta) \to 0
\end{equation}
with $M_i$ locally free. There are injective morphisms $M_{i+1} \to M_i$, and $M_i/M_{i+1}\simeq  \cO_\var $.

The following Corollary is proved as  Lemma 2 of \cite{Green1}.

\begin{corol}\label{l2}  Assume Lemma  \ref{l1} holds. If $\eta$ is  (-1)-regular, 
$H^q(\var,\Lambda^pM_i(k \eta)) = 0$ for   where $0 \le i \le c=\dim S_\beta-\dim W$, $q\ge 1$, $p\ge 1$,
 $k+q\ge p+i$.
  If  $\eta$ is 0-regular, one has the same vanishing for $q\ge 1$, $p\ge 1$,
 $k+q\ge p+i+1$.
\end{corol}

\begin{prop} \label{koszul}  Under the assumption listed at the beginning of Section \ref{preliminary}, let $W$ be a base point free subspace of $S_\beta = H^0(\var,\cO_\var(\beta))$.
The map
\begin{equation}
\label{sk}
W \otimes S_{\beta-\beta_0}\to   S_{2
\beta-\beta_0}
\end{equation}
is surjective if  $\eta$ is (-1)-regular and  $\operatorname{codim} W  \le n$.

If $\eta$ is 0-regular but not (-1)-regular, then the map \eqref{sk} is surjective if $\operatorname{codim} W  \le n -1$.

\end{prop}

\begin{proof} From equation \eqref{M1} for $i=c$ we obtain
$$W \otimes S_{\beta-\beta_0} \to   S_{2\beta-\beta_0} \to H^1(\var,M_c(\beta-\beta_0))\,.$$
Since $\beta-\beta_0=n\eta$ and $\eta$ is  (-1)-regular,  the group on the right   vanishes by Corollary \ref{l2} if $n + 1 \ge 1  + c$, i.e., if $ \operatorname{codim} W   \le n  $ (note indeed that
$\operatorname{codim} W = c)$. If $\eta$ is 0-regular, we obtain the condition $\operatorname{codim} W  \le n -1$.
\end{proof}

\subsection{Codimension of the Noether-Lefschetz loci} \label{codimensionNL} We can now prove the lower bounds
on the codimension of the Noether-Lefschetz loci.

\begin{thm}\label{main}Let $\var$ be a    simplicial toric   threefold, $\eta$  an ample  {primitive} Cartier class, $\beta_0=-K_\var$  and $\beta\in \operatorname{Pic}(\var)$    an ample Cartier class that satisfies $\beta - \beta_0 = n\eta$ for some $n \geq 0$.
Assume  that the multiplication morphism $S_\beta\otimes S_{k\eta} \to S_{\beta+k\eta}$
is surjective and also  that  $h^1(\Pt, \cO_\Pt (\beta - \eta ))=h^2(\Pt, \cO_\Pt (\beta - \eta ))=h^2(\Pt, \cO_\Pt (\beta - 2\eta ))=0$. \\
If $\eta$ is (-1)-regular, 
then
\begin{equation}\label{inequalityA3}\operatorname{codim} U_{\eta}(n)  \ \geq  \ n+1.\end{equation}
 If $\eta$ is 0-regular, then
 \begin{equation}\label{inequalityA2}\operatorname{codim} U_{\eta}(n)\geq n.\end{equation}
\end{thm}

 \begin{proof} Let $T_\beta$ be the tangent space to $U_{\eta}(n)$ at $[X]$. Since $T_{[X]}\mathcal M_\beta$ may be identified with the summand $R_\beta$ of the Jacobian ring of $X$ \cite{BaCox94}, we may take the inverse image $\tilde T_\beta$ of $T_\beta$ in the summand $S_\beta$ of the Cox ring of $X$. Now $\tilde T_\beta$ contains the summand $J_\beta$ of the Jacobian ideal of $X$, which is a base point free linear system because $X$ is quasi-smooth, hence $\tilde T_\beta$  is base point free as well.

Let us suppose that $\eta$ is (-1)-regular,   assume   that $\operatorname{codim}(U_{\eta}(n)) \le n$, and apply Proposition \ref{koszul} by taking $W=\tilde T_\beta$.
The multiplication morphism $\tilde T_\beta\otimes S_{\beta-\beta_0}\to   S_{2\beta-\beta_0}$ is surjective. The proof of  Lemma 3.7 in \cite{BG1}  gives that the Picard number of the surfaces
in $U_{\eta}(n)$ coincides with the Picard number of $\var$. But this contradicts the definition of $U_{\eta}(n)$.
The same happens  with respect to the condition $\operatorname{codim}(U_{\eta}(n)) \le n-1$
 when  $\eta$ is 0-regular.
\end{proof}

\begin{thm}\label{mainR} Let $\var$ be a   simplicial toric threefold, $\eta$  an ample {primitive} Cartier class, $\beta_0=-K_\var$, $\beta\in \operatorname{Pic}(\var)$   an ample Cartier class that satisfies $\beta - \beta_0 = n\eta$ for some $n \geq 0$. Assume  that $\beta$ is 0-regular with respect to $\eta$.
If $\eta$ is (-1)-regular,  
then
\begin{equation}\label{inequalityA3}\operatorname{codim} U_{\eta}(n) \geq n+1.\end{equation}
 If $\eta$ is 0-regular, then
 \begin{equation}\label{inequalityA2}\operatorname{codim} U_{\eta}(n)\geq n.\end{equation}
\end{thm}
\begin{proof} Follows from Theorem \ref{main}, Proposition \ref{extrareg} and  Theorem \ref{BGR}.
\end{proof}

\begin{corol}\label{Fano} Let $\var$ be a  simplicial Fano toric threefold, $\eta$  an ample  {primitive} Cartier class, $\beta_0=-K_\var$, $\beta\in \operatorname{Pic}(\var)$   a Cartier class that satisfies $\beta - \beta_0 = n\eta$ for some $n \geq 3$.  If $\eta$ is $(-1)$-regular, then
\begin{equation}\label{inequalityA3}\operatorname{codim} U_{\eta}(n) \geq n+1.\end{equation}
 If $\eta$ is 0-regular, then
 \begin{equation}\label{inequalityA2}\operatorname{codim} U_{\eta}(n)\geq n.\end{equation}
\end{corol} 

\begin{proof} By Corollary \ref{corFano} $\beta$ is 0-regular, and then the claim follows
from Theorem \ref{mainR}.
\end{proof}

\begin{rem} If $n=2$, we need to assume that $\var$ is Oda.
\end{rem}

\begin{thm}\label{mainO} Let $\var$ be a   simplicial toric Oda threefold $\eta$  an ample {primitive} Cartier class, $\beta_0=-K_\var$, $\beta\in \operatorname{Pic}(\var)$   an ample Cartier class that satisfies $\beta - \beta_0 = n\eta$ for some $n \geq 0$. If $\beta-2\eta $ is nef and $\eta$ is (-1)-regular, then
\begin{equation}\label{inequalityA3}\operatorname{codim} U_{\eta}(n) \geq n+1.\end{equation}
 If $\eta$ is 0-regular, then
 \begin{equation}\label{inequalityA2}\operatorname{codim} U_{\eta}(n)\geq n.\end{equation}
\end{thm}

\begin{proof} Follows from Theorem \ref{main} and Proposition \ref{extranef}, as the multiplication map is   surjective on Oda toric varieties.
\end{proof}

For $\var=\mathbb P^3$ and $\eta$    the hyperplane class,
the first  inequality  in the Theorems above reproduces
Green's and Voisin's result \cite{Green2,Green1,Voisin88Precision}.
 \bigskip

\subsection{Examples }\label{examples} We discuss examples of toric threefolds for which the results in 
section \ref{codimensionNL} apply.

\begin{ex}\label{0ref} Let $\var$ a toric simplicial Gorenstein threefold with nef anti-canonical bundle
 and $\eta$ an ample  Cartier class.
 Then any class $\beta=n\eta+ \beta_0$ is $m$-regular with respect to $\eta$ if $ n \geq 3-m$. Here we have used the vanishing \eqref{bottdual}. Note that when $\beta$ is 0-regular, by Proposition \ref{extrareg} the multiplication morphism $S_\beta\otimes S_{k\eta} \to S_{\beta+k\eta}$ is surjective for $k\ge 0$. 
\end{ex}

\begin{ex} \label{bLP3} Let $\widehat{\mathbb P}^3$  be the projective 3-space blown-up along a line  and $E$ the exceptional divisor. Denote by $\varpi\colon \widehat{\mathbb P}^3\to\Pt$ the blow-down morphism. By Ikeda's Theorem \ref{Ikeda}
$\widehat{\mathbb P}^3$ is an Oda simplicial toric variety. 
The nef cone of $\widehat{\mathbb P}^3$ is generated by  the   class $\eta_1 $, the pullback of a hyperplane in $\mathbb P^3$, and the   class  $\eta_2 = \eta_1 - E$.  

The anti-canonical class is
$\beta_0=3\eta_1+\eta_2$. The class $\eta=\eta_1+s \eta_2$ is ample and 0-regular for all $s\ge 1$.
The class $\beta-2\eta$ is nef for $n\ge 2 -\frac1s$. 
A very general surface  $X$  in the linear system  of the class $\beta$, with $n\ge 2 -\frac1s$,  has Picard number  $2$ by Theorem \ref{BGO}.

%
\end{ex}
 
\begin{ex}\label{p1p2} $\mathbb P^1\times \mathbb P^2$  is an Oda variety by Ikeda's Theorem \ref{Ikeda}. The nef cone is generated by the pullback of the hyperplane bundles by the natural projections, namely by the  nef classes  $H_1= \pi_1^*(\cO_{\mathbb P^2}(1))$ and $H_2= \pi_2^*(\cO_{\mathbb P^1}(1))$; the nef cone is also the cone of effective divisors. 

The ample classes $\eta=H_1+sH_2, \ s \geq 1$, are 0-regular. Moreover, $\beta -2\eta$ is nef  if  $\beta=\beta_0 + n \eta$, with $ n \geq2(s-1)/s$. A very general surface $X$ in the linear system $\beta$ has Picard number 2 by Theorem \ref{BGO}.
\end{ex}

\begin{ex} \label{smallres}
Let $\bar{\mathbb P}_\Sigma$ be the cone over a quadric surface in $\Pt$. It is a non-simplicial projective toric threefold \cite[\S 2]{OgataZhao}. We make a small resolution by adding a face which splits the cone having four edges, getting a smooth simplicial toric threefold $\var$. It is easy to see that $\var$ is quasi-Fano. Explicit toric computations show that  the nef cone is generated by $\eta_1$, the pullback of the ample generator of the class group of the quadric surface, and $\eta_2$, the divisor giving the fibration over $\mathbb P^1$ (see Figure 1). Moreover, every ample divisor of the form $\eta=\eta_1+\eta_2$ 
is 0-regular.  The anti-canonical class is $\beta_0=3\eta_1$. By Theorem \ref{Ikeda}, $\var$ is an Oda variety. With $\beta=\beta_0+n\,\eta$, the class $\beta-2\eta$ is nef for $n\ge 2$. Then  a  very general surface  $X$  in the linear system  of the class $\beta$ has Picard number  $2$ by Theorem \ref{BGO},  and  $\operatorname{codim} U_{\eta}(n) \geq n$ by Theorem \ref{mainO}.

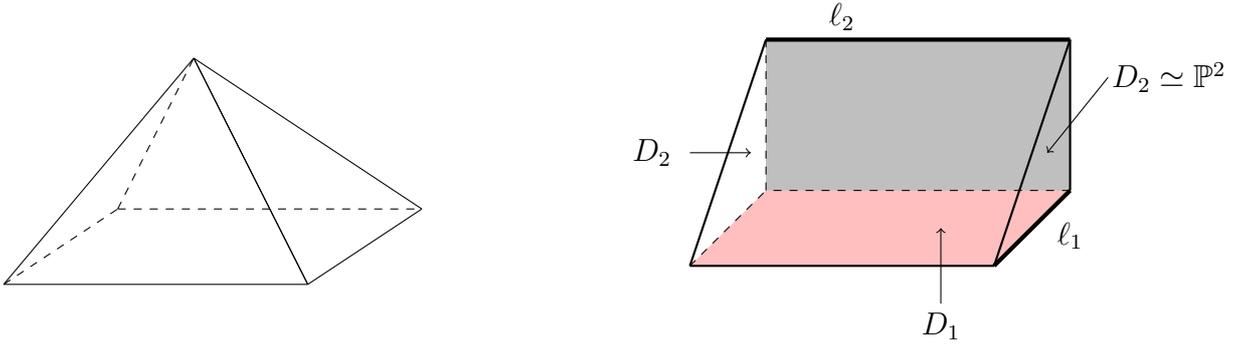
\begin{figure}\label{figure1}
\parbox{6cm}{
\begin{tikzpicture} 
\draw (0,0) -- (4,0);
\draw (0,0) -- (2.5,3);
\draw (4,0) -- (2.5,3);
\draw (4,0) -- (2.5,3);
\draw(4,0) -- (5.5,1);
\draw (2.5,3) -- (5.5,1);
\draw  [dashed] (0,0) -- (1.5,1);
\draw  [dashed] (1.5,1) -- (5.5,1);
\draw  [dashed] (1.5,1) -- (2.5,3);
 \end{tikzpicture} 
 }
 \hskip2cm
 \parbox{6cm}{
\begin{tikzpicture} 
\path [fill=lightgray] (1,1) rectangle (5,3);
\path [fill=pink] (0,0) to (4,0) to (5,1) to (1,1);
\draw [thick] (0,0) -- (4,0);
\draw  [ultra thick]  (4,0) -- (5,1);
\draw  [dashed] (0,0) -- (1,1);
\draw  [dashed] (5,1) -- (1,1);
\draw [thick] (5,1) -- (5,3);
\draw  [ultra thick]  (1,3) -- (5,3);
\node at (2,3.3) {$\ell_2$};
\draw [thick] (0,0) -- (1,3);
\draw  [dashed] (1,1) -- (1,3);
\draw [thick] (4,0) -- (5,3);
\node at (5,.4) {$\ell_1$};
\draw [->] (5.5,2.5) -- (4.7,1.5) ;
\node at (6.3,2.5) {$D_2 \simeq \mathbb P^2$};
\draw [->] (3.3,-.5) -- (3.3,0.5) ;
\node at (3.3,-.8) {$D_1$};
\draw [->] (0,1.5) -- (.8,1.5) ;
\node at (-.5,1.5) {$D_2$};

 \end{tikzpicture} 
 }

 \caption{\small Example \ref{smallres}. The picture on the left is the cone over a quadric in $\Pt$, the picture on the right is its small resolution, obtained by adding the line $\ell_2$. $D_1$, $D_2$ are divisors representing the classes $\eta_1$, $\eta_2$.}
\end{figure}
 \end{ex}

\begin{ex}\label{WP}
Let $\var = {\mathbb P}[1,1,2,2]$   and $\eta_0$ be the effective generator of the class group of $\var$, and let  $\eta$  be the (very) ample generator of the Picard group. Note that  $\eta=2\eta_0$ and $\beta_0=6\eta_0= 3 \eta$.  
The variety $\var = {\mathbb P}[1,1,2,2]$ can be realized as a quotient $\mathbb P^3/\Z_2$, and  is singular along a toric curve $C$ whose class is $\eta_0\cdot\eta_0$;
this is a line of singularities of type $A_1$.

By Theorem \ref{onlyGore}, $\var$ is and Oda variety and $\eta$ is 0-regular, but not (-1)-regular.
Let $\beta= \beta_0 + n \eta= (3+n) \eta$ with $n \ge 0$; then  $\beta - 2 \eta$ is nef. 
A very general surface  $X$  in the linear system  of the class $\beta$ has Picard number  $1$ by Theorem \ref{BGO},  and 
$\operatorname{codim} U_{\eta}(n) \geq n$ by Theorem \ref{mainO}.
\end{ex}

\bigskip
 
\section{Components of small codimension}\label{smallCod}

\subsection{Lines and surfaces containing a line}
In the previous Section we have shown that if $\var$ is an Oda variety and $\eta$ is (-1)-regular,
the inequalities
\begin{equation}\label{sharpA3} n + 1 \leq \operatorname{codim}{U}_\eta (n) \leq h^0(\var, \cO_\var (n \eta))
\end{equation} 
hold, while if $\eta$ is only 0-regular,
\begin{equation}\label{sharpA2} n \leq \operatorname{codim}{U}_\eta (n) \leq h^0(\var, \cO_\Sigma(n \eta)).\end{equation}

 In the case of $\mathbb P^3$  it was conjectured in
\cite{Green2} and proved in \cite{Green1, Green2, VoisinComposantesDeNL}, that the components of $\mathcal M_\beta$ of codimension $n+1$  are those
 whose points correspond to surfaces containing a line. In this section we define a notion of ``line''
 with respect to an ample Cartier class in a toric variety, and study the families of surfaces which contain a line in the Examples of the previous section. We show that  such components have also codimension $n+1$.

 \begin{defin} Let $\var$ be a simplicial toric threefold, equipped with a 0-regular ample Cartier class $\eta$. A line $L$ in $\var$ is a smooth rational curve such that $\eta\cdot L =1$. 
 \end{defin} 
 
A line $L$ is therefore a point in the Hilbert schemes of curves  $Hilb ^{\var}$ in $\var$   with Hilbert polynomial $p(k)=k+1$ (with respect to the polarization $\eta$); this Hilbert scheme is in general reducible.
 
 Let $\beta=\beta_0 +n\eta$, with $n \ge 0$. Assume that $\beta$ is ample, 
and let $\mathcal S_{1,\eta}(n)$ be the family of surfaces in $\var$ of class $\beta$ that contain a line.
For a line $L$ in $\var$, denote $\mathcal I_L$ its ideal sheaf.

 \begin{prop} \label{codimvan} Assume that all lines $L$ in $\var$ verify the condition $H^1(\var,\mathcal I_L(\beta))=0$. Then
 $$ \operatorname{codim} \mathcal S_{1,\eta}(n) = n + 1 + \beta_0\cdot L - \dim _{[L]}Hilb^{\var}.$$  \end{prop}


\begin{proof}
Let $\mathcal S_L\subset H^0(\var,\mathcal O_\var(\beta))$ be the family of surfaces in $\var$ of degree $\beta$ that contain a fixed line $L$.
From the exact sequence
$$ 0 \to \mathcal I_L(\beta) \to \mathcal O_\var(\beta) \to \mathcal O_L(\beta) \to 0, $$
since  $ H^1(\var,\mathcal I_L(\beta))=0$,
we see that the codimension of $\mathcal S_L$ is the dimension of 
$H^0(L,\mathcal O_L(\beta))$. 
This is 
\begin{equation}\dim  H^0(L,\mathcal O_L(\beta)) = 
\dim H^0(\mathbb P^1,\mathcal O_{\mathbb P^1}(\beta_0\cdot L+ n) )= \beta_0\cdot L+ n + 1.
 \end{equation}
 Now we have to vary the line $L$; the effect of this is to subtract from the result 
the   dimension of the Hilbert scheme of curves with Hilbert polynomial $p(k)=k+1$. \end{proof}
 
 (One can note that the general surface in $\mathcal S_{1,\eta}(n) $  does not contain two lines.)
 
%
%

\begin{lemma}\label{dimHilb}Assume that  $L \subset \var$ is  either in the smooth locus of $\var$ or a locally complete intersection and that   $h^1(N_{L/\var})=0$. Then  \begin{enumerate}
\item $\dim _{[L]}Hilb^{\var}= h^0(N_{L/\var})$.

\item If $L$ is in the smooth locus of $\var$, then $\dim _{[L]}Hilb^{\var}= h^0(N_{L/\var})= \beta_0 \cdot L$.
\end{enumerate}
 \end{lemma}
 \begin{proof}  
 (i) The statement follows from  \cite[Thm.~4.3.5]{Sernesi} and \cite[Ch.~1, Thm. ~2.8, Lemma 2.12, Prop.~2.14]{KRat}.
 
 (ii) Note that from the exact sequence 
\begin{equation}\label{seqdifferentials} 0 \to N_{L/ \var}^\ast \to \Omega^1_{\var\vert L} \to \Omega^1_L \to 0 \end{equation}
one has $\deg N_{L/ \var} = \beta_0\cdot L  - 2$, so that from Riemann-Roch one obtains
 $h^0(N_{L/ \var})= \beta_0\cdot L$.
 \end{proof}

 \begin{corol}\label{total}  If $ dim _{[L]}Hilb^{\var}= \beta_0 \cdot L$, then  $$ \operatorname{codim} \mathcal S_{1,\eta}(n) = n + 1.$$ \end{corol}
 
In our examples, the required vanishing   $H^1(\var,\mathcal I_L(\beta))=0$ in  Proposition \ref{codimvan}  will be provided  
by  the following theorem, which generalizes to the case of simplicial toric threefolds a result by Severi \cite{Severi1903} and 
 Bertram, Ein and Lazarsfeld \cite{BertramEinLazarsfeld}.
 
 \begin{thm} Let $\var$ be a simplicial projective toric threefold, and let $L$ be a toric irreducible curve in $\var$,
 which is not contained in the singular locus of $\var$, and is the intersection of two effective divisors $D_1$ and $D_2$, with
 $D_1$ Cartier. Let  $\mu\colon X \to \var$ be the blowup of $\var$ along $L$, denote by $Y$ the strict transform of $D_1$,
 and let $D$ be a divisor in $\var$ such that 
 \begin{enumerate} \item $D-D_1$ is nef;
 \item the line bundle
\begin{equation}\label{thistobenef} \mu^\ast (\cO_X(D-D_1))_{\vert Y} + N_{Y/X}\end{equation}
 is nef. \end{enumerate} Then 
 $H^i(\var,\mathcal I_L(D))=0$ for $i>0$.
 \label{vanishing}
 \end{thm}
 
 \begin{proof}
One has $\mu_\ast(\mathcal O_X(-F))\simeq \mathcal I_L$, where
 $F\subset X$ is the exceptional divisor. We see that one has indeed 
 $\mu_\ast\cO_X\simeq \cO_\var$, and 
 $R^i\mu_\ast\cO_X=0$ for $i>0$ as $\mu$ is  a toric morphism \cite[Thm.~9.2.5]{CoxLSh}.
By applying the functor $\mu_\ast$ to the fundamental exact sequence of the divisor $F$, we obtain that   $\mu_\ast(\mathcal O_X(-F))$ is a rank 1 subsheaf of $\mathcal O_\var$,
 i.e., it is the ideal sheaf of a closed subscheme $L'$ of $L$. Actually, $L'=L$; indeed, if $L'$ were properly contained in $L$, it would consist of a finite number of points, and its structure sheaf would be a torsion subsheaf of $\mu_\ast{\cO_F}_{\vert L}$. However, the latter
 is locally free, as follows from the fact that all fibres of $F \to L$ are integral rational curves, so that $F\to L$ is a flat morphism, and one can apply cohomology and base change. So $L'=L$, and  $\mu_\ast(\mathcal O_X(-F))\simeq \mathcal I_L$.
 
Then we have the exact sequence
 $$ 0 \to \mathcal I_L \to \cO_\var \to \mu_\ast\cO_F \to R^1\mu_\ast (\mathcal O_X(-F))\to  R^1\mu_\ast\cO_X =0.$$
 Since the first three terms form  an exact sequence,  we obtain   $R^1 \mu_\ast(\mathcal O_X(-F) )= 0$, and similarly $R^i \mu_\ast(\mathcal O_X(-F) )= 0$  for $i>1$.
 Now  the projection formula and the Leray spectral sequence give
 $$H^i(X,\mu^\ast\mathcal O_\var(D)\otimes \mathcal O_X(-F)) \simeq H^i(\var,\mathcal I_L(D)).$$
 
If the divisor $G = \mu^\ast(D)-F$ is nef the conclusion of the Theorem will follow from 
  Demazure's vanishing theorem  (\cite[Thm.~9.2.3]{CoxLSh}). Let $\Gamma$ be an irreducible curve in $X$.
  If $\Gamma$ is not contained in $Y$, then the $ [G] \cdot  \Gamma\ge 0$ as $D-D_1$ is nef. If $\Gamma$ is in $Y$,
  then $[G] \cdot  \Gamma\ge  0$ as the line bundle \eqref{thistobenef}  is nef.
\end{proof}

We revisit now the Examples in Section \ref{codimsec} and compute
  the codimension of the loci of surfaces containing a line.
 In all cases, the vanishing in Theorem \ref{vanishing} holds.
Moreover, we can apply it also to lines $L$ that are not toric, as a suitable automorphism can always take them to toric lines.
 
  \subsection{Examples: $\var$ smooth}   
 \begin{ex} We consider   the Example \ref{bLP3} again. 
The Mori cone of $\widehat{\mathbb P}^3$ is generated by
the  curves $\ell_1$ and $\ell_2$, where $\ell_1$ determines   the extremal contraction
$\widehat{\mathbb P}^3\to\Pt$, and $\ell_2$ determines the natural morphism $\widehat{\mathbb P}^3\to\mathbb P^1$. One has $\ell_i\cdot\eta_j=0$ if $i=j$, and $\ell_i\cdot\eta_j=1$ if $i\ne j$. The class $\eta=\eta_1+\eta_2$ is very ample, and the curves $\ell_1$ and $\ell_2$ are lines with respect to it.
Note that the Hilbert scheme of curves in $\widehat{\mathbb P}^3$ with Hilbert polynomial
$p(k)=k+1$ has two connected components, corresponding to the homology classes
of the lines $\ell_1$ and $\ell_2$. 

  The general line numerically equivalent to $\ell_2$ is cut by divisors in the classes $\eta_1$ and $\eta_2$. It is easy to see that $h^1(N_{\ell_2/\var})=0$.

  We consider the family $\mathcal S_{1,\eta}(n)_{\ell_2}$ of surfaces of class $\beta=\beta_0+n\,\eta$ (with $n\ge 0$) in 
$\widehat{\mathbb P}^3$ that contain a line homologically equivalent to $\ell_2$. By Proposition \ref{codimvan}, Lemma \ref{dimHilb} and Corollary \ref{total}
$$\operatorname{codim} \mathcal S_{1,\eta}(n)_{\ell_2}= n+1.$$
\end{ex}
 
%

\begin{ex} ($\mathbb P^1 \times \mathbb P^2$, Example \ref{p1p2})
The dual cone of effective curves (the Mori cone) is generated by the rational curves 
$\ell_1=H_1 \cdot H_2$ and $\ell_2= H_1 \cdot H_1$. The curves with classes $\ell_1$ and $\ell_2$  
are lines with respect to the ample class $\eta=H_1+H_2$.  It is easy to see that $h^1(N_{\ell_1/\var})=h^1(N_{\ell_2/\var})=0$.
The families of surfaces $\mathcal S_{1,\eta}(n)_{\ell_1}$ and $\mathcal S_{1,\eta}(n)_{\ell_2}$
containing a line numerically equivalent to $\ell_1$ and $\ell_2$, respectively, have both codimension $n+1$  by Proposition \ref{codimvan}, Lemma \ref{dimHilb} and Corollary \ref{total}. 

One can note that if we take $\eta=H_1+sH_2$, then $\mathcal S_{1,\eta}(n)_{\ell_1}$ has codimension $ns+1$, while $\mathcal S_{1,\eta}(n)_{\ell_2}$ still has   codimension $n+1$ (note indeed that $\ell_1$ is not a line).
\end{ex}

\begin{ex} (The small resolution of the cone over a quadric, see Example \ref{smallres}.)
The Mori cone is generated by the lines $\ell_1$ and $\ell_2$ (see Figure 1). $\ell_2$ is the exceptional curve of the small resolution, while $\ell_1$ determines the natural morphism to $\mathbb P^1$. Both are lines with respect to the ample class $\eta=\eta_1+\eta_2$. The general line numerically equivalent to $\ell_1$ is cut by divisors in the classes $\eta$ and $\eta_2$. It is easy to see that $h^1(N_{\ell_1/\var})=0$.  If $\beta=\beta_0+n\,\eta$ with $n\ge 2$, we consider the family $\mathcal S_{1,\eta}(n)_{\ell_1}$ of surfaces in $\var$ of class $\beta$ which contain a line homologically equivalent to $\ell_1$.  
Proposition \ref{codimvan}, Lemma \ref{dimHilb} and Corollary \ref{total} yield $$\operatorname{codim} \mathcal S_{1,\eta}(n)_{\ell_1}= n+1.$$
\end{ex}
 
\subsection{$\pmb{\var =  \mathbb P [1,1,2,2]}$}\  \\[3pt]
We revisit also Example \ref{WP}. The numerical class of a line $L$ is given by $\eta\cdot\eta_0$.  We set as usual $\beta=\beta_0+n\eta$, with $n\ge 0$.  We prove that $\dim Hilb^{\var} = \beta_0\cdot L=3$, so that   the family $\mathcal S_1$ of surfaces in $\var$ that contain a line has codimension $n+1$.
Note that since $L$ is a complete intersection,   Lemma \ref{dimHilb} applies if $h^1(N_{L/\var})=0$.
 In fact we have:
 
\begin{lemma}
$\deg N_{L/\var} =1$ and $h^1(N_{L/\var})=0$. 
\end{lemma}
\begin{proof} We consider the exact sequence
\begin{equation}\label{normal} 0 \to K \to {\mathcal I_L}/{\mathcal I_L}^2 \xrightarrow{d\ } \Omega^1_{\var\vert L} \to \Omega^1_L \to 0  \end{equation}
where $\Omega^1_\var$ is the sheaf of K\"ahler differentials on $\var$.
$K$ is a torsion sheaf concentrated at the intersection point $P$ of $L$ with the curve $C$ of singularities of $\var$. 
Moreover, 
$K$ is the torsion of ${\mathcal I_L}/{\mathcal I_L}^2$. Since $L$ is a smooth curve, ${\mathcal I_L}/{\mathcal I_L}^2$ splits as 
$$ {\mathcal I_L}/{\mathcal I_L}^2 = N_{L/\var}^\ast\oplus K,$$
indeed (see e.g.~\cite{Sernesi}), by dualizing \eqref{normal} we obtain 
$$ 0 \to T_L \to ( \Omega^1_{\var\vert L})^\ast \to N_{L/\var}  \to 0\,. $$
Hence we have
$$\deg N_{L/\var} = 3- 2 = 1.$$
\end{proof}
\begin{lemma} $H^0(L,\Omega^1_{\var\vert L}) = 0$.
\end{lemma}
\begin{proof} Let $\pi\colon \mathbb P^3 \to \mathbb P^3/\Z_2 = \var$. From \cite{Dolg-weighted, Knighten} we have 
$$ \pi_\ast^{\mathbb Z_2} \,\Omega^1_{\mathbb P^3} \simeq \widehat\Omega^1_{\var}$$
where $\widehat\Omega^1_{\var}$ is the sheaf of Zariski differentials (see section \ref{Zariski}). 
Let $\tilde L = \pi^{-1}(L)$. 
Note that $H^0(\tilde L, \Omega^1_{\mathbb P^3\vert \tilde L} ) =0$, as follows from the Euler sequence of $\mathbb P^3$ restricted to $\tilde L$.
Also, we have
$$ H^0(L, \widehat\Omega^1_{\var\vert L}) = H^0(L, \pi_\ast^{\mathbb Z_2} \Omega^1_{\mathbb P^3\vert \tilde L} )
=H^0(\tilde L, \Omega^1_{\mathbb P^3})^{\Z_2}. $$

Since  $H^0(\tilde L, \Omega^1_{\mathbb P^3})^{\Z_2}=0$, we have
\begin{equation}  H^0(L, \widehat\Omega^1_{\var\vert L}) = 0 .\label{nosec} \end{equation}

Let $i\colon U \to \var$ be the embedding of the smooth locus of $\var$. By definition,
$\widehat\Omega^1_\var = i_\ast i^\ast \Omega^1_\var$, so that we have an exact sequence
$$ 0 \to J \to  \Omega^1_\var \to \widehat\Omega^1_\var \to Q \to 0.$$
However, since $J$ is supported on the singular locus of $\var$, and $ \Omega^1_\var $ is torsion-free one has $J=0$ \cite[Prop.~9.7 and Cor.~9.8]{Kunz},
\cite[Thm.~3]{Knighten} (see also \cite[Lemma 1.8]{SteenVanishing}  and  \cite{Greb-Rollenske}).
The sheaf $Q$ is supported on the singular locus as well. 
 By restricting to $L$ we obtain
$$ \operatorname{Tor}_1(Q,\cO_L) \to \Omega^1_{\var \vert L} \to  \widehat\Omega^1_{\var \vert L}  $$
Now $ \operatorname{Tor}_1(Q,\cO_L)$ is concentrated on $P$, and since $\Omega^1_{\var \vert L} $
is an extension of the locally free sheaves $N_{L/\var}^\ast$ and $\Omega^1_L$, it is locally free as well, hence $ \operatorname{Tor}_1(Q,\cO_L)$ maps to zero, and $ \Omega^1_{\var \vert L} $ injects into $\widehat\Omega^1_{\var \vert L}$. From \eqref{nosec} we get
$$H^0(L,\Omega^1_{\var \vert L} )= 0\,.$$
\end{proof}
%


Now we consider the exact sequence
$$ 0 \to N_{L/\var}^\ast (-2) \to \Omega^1_{\var \vert L} (-2) \to \cO_L(-4) \to 0 $$
whence we get $h^0(N_{L/\var}^\ast(-2))=0$. By Serre duality, $h^1(N_{L/\var})=0$, so that
$$h^0(N_{L/\var})= \chi(N_{L/\var})= 3.$$

So  also in this case we have shown that the family $ \mathcal S_{1,\eta}(n)_L$ has codimension $n+1$.

\frenchspacing\bigskip

\def\cprime{$'$}

\end{document}